\documentclass[12pt,reqno]{amsart}

\usepackage[pdftex]{graphicx} 
\usepackage[english]{babel} 
\usepackage[pdftex,linkcolor=black,pdfborder={0 0 0}]{hyperref} 
\usepackage{calc} 
\usepackage{enumitem} 
\usepackage{comment}

\usepackage{import}

\usepackage{adjustbox}

\usepackage{amsmath, amssymb, amsthm, hyperref}
\usepackage[all]{xy}
\usepackage[pdftex]{graphicx}
\usepackage{color}
\usepackage{cite}
\usepackage{url}
\usepackage{indent first}
\usepackage[labelfont=bf,labelsep=period,justification=raggedright]{caption}
\usepackage[english]{babel}
\usepackage[utf8]{inputenc}
\usepackage[colorinlistoftodos]{todonotes}
\usepackage{tkz-fct}
\usepackage{tikz}
\usetikzlibrary{calc}
\usepackage{comment}

\usepackage{mathtools} 
\mathtoolsset{showonlyrefs} 

\usepackage{multicol}
\PassOptionsToPackage{dvipsnames,svgnames}{xcolor}
\usepackage{textcomp}
\usepackage{sistyle}
\SIthousandsep{,}

\usepackage{fullpage}

\usepackage{amsmath, amssymb}
\usepackage{import}

\usepackage{tikz}
\usetikzlibrary{decorations.markings}

\theoremstyle{plain}
\newtheorem{theorem}{Theorem}[section]

\newtheorem{lemma}[theorem]{Lemma}

\theoremstyle{definition}
\newtheorem*{remark}{Remark}

\theoremstyle{definition}

\usepackage{pgfplots}
\usepackage{amsmath}
\pgfplotsset{compat=1.18}

\renewcommand{\Re}{\operatorname{Re}}

\newcommand{\logb}[1]{\log{\left(#1\right)}}

\begin{document}

\title{On the error in the prime number theorem in short intervals}

\author[E.~S.~Lee]{Ethan~Simpson~Lee}
\address{University of the West of England, School of Computing and Creative Technologies, Coldharbour Lane, Bristol, BS16 1QY} 
\email{ethan.lee@uwe.ac.uk}
\urladdr{\url{https://sites.google.com/view/ethansleemath/home}}

\maketitle

\begin{abstract}
Assuming the Riemann Hypothesis, we derive explicit bounds for the error terms in short interval analogues of the prime number theorem using a new smoothing argument. Our results improve upon earlier bounds in both constant terms and applicable ranges. In addition, we apply our bounds to establish sharper conditional bounds for a broad class of weighted sums over primes in a short interval.
\end{abstract}

\section{Introduction}

Throughout this paper, $(x,x+h]$ such that $h = o(x)$ as $x\to\infty$ is called a short interval and $p$ denotes a prime number. In addition, let
\begin{equation*}
    \pi(x) = \sum_{p\leq x} 1 , \quad
    \theta(x) = \sum_{p\leq x} \log{p} , \quad\text{and}\quad
    \psi(x) = \sum_{n\leq x} \Lambda(n) ,
\end{equation*}
where $\Lambda$ is the von Mangoldt function. The asymptotic results $\pi(x) \sim x/\log{x}$, $\theta(x) \sim x$, and $\psi(x) \sim x$ are collectively called the prime number theorem. Each statement in the prime number theorem is equivalent and the error in $\psi(x) \sim x$ is directly linked to the distribution of zeros of the Riemann zeta-function. 

Between 1962 and 1976, Rosser and Schoenfeld wrote a series of foundational papers \cite{Rosser,RosserSchoenfeld,Schoenfeld}, which provide explicit descriptions of the error in each statement in the prime number theorem. Their results continue to be extensively applied in the literature, and refinements to their results are an active area of research. For examples, see \cite{Broadbent, PlattTrudgian, BPT_mean_square}. By extension, we also expect explicit versions of short interval analogues of the prime number theorem to be widely applicable. In particular, such results allow us to effectively study the local distribution of primes. 

The primary objective of this paper is to establish new explicit descriptions for the error in short interval analogues of each statement in the prime number theorem. Our main results are presented in Theorem \ref{thm:main_pnt} and Theorem \ref{thm:prime_sum_bounder}. The Riemann Hypothesis (RH) is assumed throughout, to obtain the strongest outcomes. This is a standard assumption in the literature, and there is a lot of evidence to support it (see \cite{PlattTrudgianRH}).



For any $h = o(x)$ which is ``large enough'', short interval analogues of the prime number theorem have the form
\begin{equation}\label{eqn:PNT_in_SI_forms}
    \pi(x+h) - \pi(x) \sim \frac{h}{\log{x}} , \quad
    \theta(x+h) - \theta(x) \sim h, \quad\text{and}\quad 
    \psi(x+h) - \psi(x) \sim h .
\end{equation}
The asymptotic results in \eqref{eqn:PNT_in_SI_forms} are closely related, so progress toward one result translates into progress on all three results in \eqref{eqn:PNT_in_SI_forms}. To clarify what we mean by ``large enough'', we recall results from Heath-Brown \cite{HeathBrown}, Maier \cite{Maier}, and Selberg \cite{Selberg} on the asymptotic involving $\pi(x)$ in \eqref{eqn:PNT_in_SI_forms}. Heath-Brown proved this result (unconditionally) when $h = x^{\varepsilon}$ and $\varepsilon > 7/12$, Maier proved that this result is false if $h = (\log{x})^\lambda$ for any $\lambda > 1$, and Selberg proved this result assuming the RH is true and $h/\sqrt{x}\log{x} \to\infty$ as $x\to\infty$. In light of these results, we observe that $h = x^{\varepsilon}$ such that $\varepsilon > 1/2$ is ``large enough''.

Before we introduce our first result, recall that if $x\geq 73.2$ and the RH is true, then a straightforward bound for the approximation involving $\psi(x)$ in \eqref{eqn:PNT_in_SI_forms} is 
\begin{align*}
    \left| \psi(x+h) - \psi(x) - h \right|
    &\leq \frac{\sqrt{x} (\log{x})^2}{8\pi} \left( 1 + \frac{\sqrt{x + h}}{\sqrt{x}} \left(\frac{\logb{x+h}}{\log{x}}\right)^2\right) . 
\end{align*}
This follows from a famous result from Schoenfeld \cite[Thm.~10]{Schoenfeld} on the error in $\psi(x) \sim x$. This explicit bound is one of the best we can prove today for $h$ closer to $x$, although sharper bounds can be proved for smaller choices of $h$. In particular, we prove Theorem \ref{thm:main_pnt}, which presents a sharper bound when $\sqrt{x}\log{x} \leq h \leq x^{\alpha}$ for any 
\begin{equation*}
    \alpha \leq 
    \begin{cases}
        0.64951\dots &\text{on $x \geq e^{20}$,} \\
        0.71705\dots &\text{on $x \geq e^{40}$.}
    \end{cases}
\end{equation*}

\begin{theorem}\label{thm:main_pnt}
Suppose that the RH is true, $x \geq e^{20}$, and
\begin{equation}\label{eqn:L_def}
    L 
    = \frac{2}{\pi + \sqrt{\pi^2 + 4\pi}} + \frac{1}{\pi} \left( 2 + \logb{\frac{\pi ( \pi + \sqrt{\pi^2 + 4\pi} + 2 )}{4}}\right) = 1.54262\dots .
\end{equation}
If $\sqrt{x}\log{x} \leq h \leq x^{3/4}$, then 
\begin{equation}\label{eqn:SI_result_psi_explicit}
    \left| \psi(x+h) - \psi(x) - h \right| < \left(\frac{1}{\pi} \logb{\frac{h}{\sqrt{x}\log{x}}} + L + \frac{1.022}{\log{x}}\right) \sqrt{x}\log{x} 
\end{equation}
and 
\begin{equation}\label{eqn:SI_result_theta_explicit}
    \left| \theta(x+h) - \theta(x) - h \right| < \left(\frac{1}{\pi} \logb{\frac{h}{\sqrt{x}\log{x}}} + L + \frac{1.046}{\log{x}}\right) \sqrt{x}\log{x} .
\end{equation}
\end{theorem}

Prior to this paper, Cully-Hugill and Dudek proved a related result to \eqref{eqn:SI_result_psi_explicit} only in \cite[Thm.~1]{CullyHugillDudek}. In particular, if $\sqrt{x}\log{x} \leq h \leq x^{3/4}$, $\log{x} \geq 40$, and the RH is true, then they proved\footnote{There was a typo in \cite{CullyHugillDudek} that affected some of their constants; Cully-Hugill has corrected the typo and updated their result in her PhD thesis \cite{CullyHugillThesis}. The bound \eqref{thm:CH_D} is the result that is stated in her thesis.}
\begin{equation}\label{thm:CH_D}
    |\psi(x+h) - \psi(x)-h|< \left(\frac{1}{\pi} \logb{\frac{h}{\sqrt{x}\log x}} + 2.167\right)\sqrt{x}\log{x} .
\end{equation} 
Our result \eqref{eqn:SI_result_psi_explicit} holds for a broader range of $x$ and offers constant improvement as $x\to\infty$. To exemplify this, if $\log{x}\geq 40$, then Theorem \ref{thm:main_pnt} demonstrates that their constant $2.167$ can be refined to $1.56817\ldots$. Our main innovation, which all of our improvements can be attributed to, is the implementation of a new, different smoothing method. 

Next, suppose that $f$ is any decreasing, non-negative, continuous function. Applying the explicit bounds in \eqref{eqn:SI_result_theta_explicit}, we prove the following general-purpose result, which provides sharper conditional approximations for any sum of $f(p)$ over primes $p$ in a short interval. These bounds have significant potential for utility in a range of intermediate calculations arising in analytic number theory.

\begin{theorem}\label{thm:prime_sum_bounder}
Suppose the RH is true and $\sqrt{x}\log{x} \leq h \leq x^{3/4}$. If $\log{x}\geq 20$, then
\begin{align*}
    \Bigg| \sum_{x < p\leq x+h} f(p) - \int_{x}^{x+h} \frac{f(t)}{\log{t}}\,dt \Bigg| 
    &\leq \left(\frac{1}{\pi} \logb{\frac{h}{\sqrt{x}\log{x}}} + L + \frac{1.046}{\log{x}}\right) f(x) \sqrt{x} \\
    &\qquad\qquad + \left(1 + \frac{h}{x}\right)^{\frac{5}{2}} \frac{f(x) \sqrt{x}\log{x}}{8\pi} \\
    &\qquad\qquad - \frac{1}{8\pi} \int_{x}^{x+h} \sqrt{t}\log{t} \left(f'(t) - \frac{f(t)}{t\log{t}}\right)\,dt .
\end{align*}
\end{theorem}

We demonstrate the extra precision one can access using Theorem \ref{thm:prime_sum_bounder} by means of an example. That is, assuming the RH is true, $f(p) = 1/p$, $\log{x} \geq 20$, and $\sqrt{x}\log{x} \leq h \leq x^{3/4}$, we bound the objective sum two ways and compare outcomes. First, two applications of an explicit version of Mertens' second theorem (e.g., see \cite[(6.21)]{Schoenfeld} or \cite[Cor.~1.3]{LeeNosal}) imply
\begin{equation}\label{eqn:MT2_SI_explicit_Schoenfeld}
    \sum_{x < p\leq x+h} \frac{1}{p} = \int_{x}^{x+h} \frac{dt}{t\log{t}} 
    + O\Bigg( \frac{6\log{x}}{8\pi\sqrt{x}} \Bigg) .
\end{equation}
Second, Theorem \ref{thm:prime_sum_bounder} implies 
\begin{equation}\label{eqn:MT2_SI_explicit}
    \sum_{x < p\leq x+h} \frac{1}{p} = \int_{x}^{x+h} \frac{dt}{t\log{t}} 
    + O\Bigg( \frac{3\log{x}}{8\pi\sqrt{x}} \Bigg) .
\end{equation}
A detailed proof of \eqref{eqn:MT2_SI_explicit} is given in Section \ref{sec:proof2}. Clearly, the leading constant in the error in \eqref{eqn:MT2_SI_explicit_Schoenfeld} is a factor of two worse than the leading constant in \eqref{eqn:MT2_SI_explicit}, which makes explicit the extra strength achieved by Theorem \ref{thm:prime_sum_bounder}. 

\begin{remark}
It is conjectured that the approximations in \eqref{eqn:PNT_in_SI_forms} are true with $h = x^{\varepsilon}$ for any $0 < \varepsilon < 1$. While this remains an open problem, Bank, Bary-Soroker, and Rosenzweig have proved the function field analogue of this result in \cite{BankBarySorokerRosenzweig}. Results in function fields often mirror deep conjectures in number fields, so their proof provides strong evidence that the corresponding number-theoretic conjecture is true.
\end{remark}

\subsection*{Structure}

The remainder of this paper is organised as follows. In Section \ref{sec:explicit_formula}, we prove Theorem \ref{thm:approximation}, which approximates weighted sums that over-or-underestimate $\psi(x+h) - \psi(x)$ using a smoothing argument. This is by far the most technical aspect of the paper. In Section \ref{sec:proof_pnt_si}, we prove Theorem \ref{thm:main_pnt} using Theorem \ref{thm:approximation}. In Section \ref{sec:proof2}, we prove Theorem \ref{thm:prime_sum_bounder} using \eqref{eqn:SI_result_theta_explicit} in standard arguments.

\section{Explicit approximation for a weighted sum}\label{sec:explicit_formula}

Let
\begin{equation*}
    h = o(x),\quad 0 < \delta \leq \frac{h}{2}, \quad h_0 = \frac{h}{x}, \quad \delta_0 = \frac{\delta}{x} ,
\end{equation*}
and $\widehat{w}$ be a smooth function such that $0 \leq \widehat{w}(t) \leq 1$ for $0 \leq t \leq 2$, $\widehat{w}(1) = 1$, and $\widehat{w}^{(k)}(0) = \widehat{w}^{(k)}(2) = 0$ for any $k \in \{0,1,2\}$. In particular, we will choose
\begin{equation}\label{eqn:w_hat_choice}
    \widehat{w}(t) = \frac{\cos(\pi(t-1)) + 1}{2} . 
\end{equation}
Next, let
\begin{equation*}
    w_{\pm}(t) = 
    \begin{cases}
        \widehat{w}\left(\frac{t-\alpha_{\pm}+\delta_0}{\delta_0}\right) & \text{if } \alpha_{\pm} - \delta_0 < t < \alpha_{\pm},\\
        1 & \text{if } \alpha_{\pm} \leq t \leq \beta_{\pm},\\
        \widehat{w}\left(\frac{t-\beta_{\pm}+\delta_0}{\delta_0}\right) & \text{if } \beta_{\pm} < t < \beta_{\pm} +\delta_0,\\
        0 & \text{otherwise},
    \end{cases}
\end{equation*}
in which 
\begin{equation*}
    \alpha_{\pm} = 
    \begin{cases}
        1 &\text{if $+$,}\\
        1 + \delta_0 &\text{if $-$,}
    \end{cases}
    \quad\text{and}\quad
    \beta_{\pm} = 
    \begin{cases}
        1 + h_0 &\text{if $+$,}\\
        1 + h_0 - \delta_0 &\text{if $-$.}
    \end{cases}
\end{equation*}
The weight function $w_{\pm}(t)$ is visualised in Figure \ref{fig:w_pm}. In this section, we bound the error in the approximation 
\begin{equation*}
    \Omega_{\pm}(x) = \sum_{n} w_{\pm}(n/x) \Lambda(n) \sim h ,
\end{equation*}
for appropriate $h$. Our result is presented below.

\begin{theorem}\label{thm:approximation}
Recall the definition of $L$ from \eqref{eqn:L_def}. If the RH is true, $\log{x} \geq 20$, and $\sqrt{x}\log{x} \leq h \leq x^{3/4}$, then 
\begin{equation}\label{eqn:recall_me}
    |\Omega_{\pm}(x) - h |
    < \left(\frac{1}{\pi} \logb{\frac{h}{\sqrt{x}\log{x}}} + L + \frac{1.02164}{\log{x}}\right) \sqrt{x}\log{x} . 
\end{equation}
\end{theorem}

\begin{figure}
\begin{center}
\begin{tikzpicture}
  \begin{axis}[
    axis lines=middle,
    xlabel={$t$},
    ylabel={$w_{\pm}(t)$},
    ymin=-0.1, ymax=1.1,
    xmin=0.3, xmax=2.2,
    xtick={0.5,1,1.5,2.0},
    xticklabels={$\,\alpha_\pm - \delta_0$, $\alpha_\pm$, $\beta_\pm$, $\beta_\pm + \delta_0$},
    ytick={0,1},
    domain=0.0:3.0,
    samples=400,
    smooth,
    thick
  ]
  
  \addplot [
    blue
  ] 
  ({x}, {
    (x >= 0.5 && x < 1.0) * (cos(deg(pi*(2*(x - 0.5) - 1))) + 1)/2 +
    (x >= 1.0 && x <= 1.5) * 1 +
    (x > 1.5 && x < 2.0) * (cos(deg(pi*(2*(x - 1.5)))) + 1)/2
  });

  \end{axis}
\end{tikzpicture}
\end{center}
    \caption{The weight function $w_{\pm}(t)$.}
    \label{fig:w_pm}
\end{figure}

To prove this result requires several technical arguments, so we break this process into four distinct steps, split across four sections. In Section \ref{ssec:auxiliary_bounds}, we establish properties of the inverse Mellin transform of $w_{\pm}$ that will be important throughout. In Section \ref{ssec:explicit_formula}, we use these insights to establish an explicit formula, which relates $\Omega_{\pm}(x)$ to a sum over the non-trivial zeros of the Riemann zeta-function $\zeta(s)$. In Section \ref{ssec:approximate_formula}, we import auxiliary results to bound the sum over zeros and establish a parametrised upper bound for $|\Omega_{\pm}(x) - h|$; this result is presented in Lemma \ref{lem:paramers_everywhere}. In Section \ref{ssec:final_steps}, we make optimised choices for these parameters and deduce Lemma \ref{lem:Omega_omega_approximation}. Using Lemma \ref{lem:Omega_omega_approximation}, we prove Theorem \ref{thm:approximation} at the end of Section \ref{ssec:final_steps}.

\begin{remark}
Originally, $\widehat{w}(t) = t^m (2-t)^m$ was chosen for some real parameter $m>1$. After applying this choice and optimising for $m$ at the end, this closely resembled the trigonometric function in \eqref{eqn:w_hat_choice}, which is more elegant and allowed us to simplify our arguments. 
\end{remark}

\subsection{Step 1: The inverse Mellin transform}\label{ssec:auxiliary_bounds}

Henceforth, let $\varrho = \beta + i\gamma$ denote the complex zeros of the Riemann zeta-function $\zeta(s)$. Further, recall the inverse Mellin transform of $w_{\pm}$ is 
\begin{equation*}
    W_{\pm}(s) = \int_0^\infty y^{s-1} w_{\pm}(y) \,dy ,
\end{equation*}
and the inverse Mellin transform formula tells us
\begin{equation}\label{eqn:InvMellinTransformW}
    w_{\pm}(t) = \frac{1}{2\pi i} \int_{2-i\infty}^{2+i\infty} W_{\pm}(s) t^{-s} \,ds,
\end{equation}
which is admissible because $W_{\pm}(s)$ is analytic on $\Re{s} > 0$. The following lemma provides important bounds on $W_{\pm}(s)$.

\begin{lemma}\label{lem:properties}
First, we have
\begin{equation*}
    |x W_{\pm}(1) - h | \leq \delta .
\end{equation*}
Second, if the RH is true, $j\in\{0,1,2,3\}$, and $\varrho = \beta + i\gamma$ is any non-trivial zero of $\zeta(s)$, then
\begin{equation*}
    |W_{\pm}(\varrho)| \leq \frac{\delta_0^{1-j} \Phi_j}{|\gamma|^{j}}
\end{equation*}
such that
\begin{equation*}
    \Phi_j = 
    \begin{cases}
        \frac{h}{\delta} + \frac{1}{\sqrt{1-\delta_0}} &\text{if $j = 0$,} \\
        2 (1 + h_0 + \delta_0)^{\frac{1}{2}} &\text{if $j = 1$,} \\
        2\pi (1 + h_0 + \delta_0)^{\frac{3}{2}}  &\text{if $j = 2$.} 
    \end{cases}
\end{equation*}
Third, if $\ell$ is any positive integer, then
\begin{equation*}
    |W_{+}(-2\ell)| \leq \frac{(1-\delta_0)^{-2\ell}}{\ell} .
\end{equation*}
\end{lemma}

\begin{proof}
To begin, note that our choice \eqref{eqn:w_hat_choice} implies
\begin{equation}\label{eqn:integral_comps}
\begin{split}
    \int_{0}^2 \widehat{w}(u)\,du = 1, \quad
    \int_{0}^2 |\widehat{w}^{(1)}(u)|\,du = 2, \quad
    \int_{0}^2 |\widehat{w}^{(2)}(u)|\,du &= 2\pi .
\end{split}
\end{equation}
Using the substitutions $t = \delta_0 u + \alpha_{\pm} - \delta_0$ and $t = \delta_0 u + \beta_{\pm} - \delta_0$ respectively, we have
\begin{align*}
    W_{\pm}(1) 
    = \beta_{\pm} - \alpha_{\pm} + \int_{\alpha_{\pm}-\delta_0}^{\alpha_{\pm}} w_{\pm}(t)\,dt + \int_{\beta_{\pm}}^{\beta_{\pm}+\delta_0} w_{\pm}(t)\,dt 
    &= \beta_{\pm} - \alpha_{\pm} + \delta_0 \int_0^2 \widehat{w}(u)\,du \\
    &= \beta_{\pm} - \alpha_{\pm} + \delta_0 .
\end{align*}
It follows that $W_{+}(1) = h_0 + \delta_0$ and $W_{-}(1) = h_0 - \delta_0$, hence
\begin{equation*}
    |W_{\pm}(1) - h_0 | \leq \delta_0 .
\end{equation*}

To prove the next statement, observe that the chain rule implies
\begin{equation*}
    \delta_0 w_{\pm}^{(1)}(t) = 
    \begin{cases}
        \widehat{w}^{(1)}(f_1(t)) &\text{if $\alpha_{\pm}-\delta_0 \leq t \leq \alpha_{\pm}$}\\
        \widehat{w}^{(1)}(f_2(t)) &\text{if $\beta_{\pm} \leq t \leq \beta_{\pm} + \delta_0$}
    \end{cases}
    \quad\text{where}\quad
    f_i(t) = 
    \begin{cases}
        \frac{t-\alpha_{\pm} +\delta_0}{\delta_0} &\text{if $i=1$,}\\
        \frac{t-\beta_{\pm} +\delta_0}{\delta_0} &\text{if $i=2$.}
    \end{cases}
\end{equation*}
It follows from integration by parts and the substitutions $t = \delta_0 u + \alpha_{\pm} - \delta_0$, $t = \delta_0 u + \beta_{\pm} - \delta_0$ that
\begin{align*}
    W_{\pm}(\varrho) 
    &= - \frac{\int_{0}^{1} (\delta_0 u + \alpha_{\pm} - \delta_0)^{\varrho} \widehat{w}^{(1)}(u)\,du + \int_{1}^{2} (\delta_0 u + \beta_{\pm} - \delta_0)^{\varrho} \widehat{w}^{(1)}(u)\,du}{\varrho} \\
    &= \frac{\delta_0^{-1} \left(\int_{0}^{1} (\delta_0 u + \alpha_{\pm} - \delta_0)^{\varrho + 1} \widehat{w}^{(2)}(u)\,du + \int_{1}^{2} (\delta_0 u + \beta_{\pm} - \delta_0)^{\varrho + 1} \widehat{w}^{(2)}(u)\,du \right)}{\varrho (\varrho + 1)} .
\end{align*}
Moreover, the RH implies
\begin{align*}
    |W_{\pm}(\varrho)| 
    &= \left|\int_{\alpha_{\pm}}^{\beta_{\pm}} t^{\varrho-1}\,dt + \Bigg(\int_{\alpha_{\pm}-\delta_0}^{\alpha_{\pm}} + \int_{\beta_{\pm}}^{\beta_{\pm}+\delta_0}\Bigg) t^{\varrho-1} \widehat{w}(t)\,dt \right| \\
    &\leq \int_{\alpha_{\pm}}^{\beta_{\pm}} t^{-\frac{1}{2}}\,dt + \delta_0\left|\int_{0}^{1} (\delta_0 u + \alpha_{\pm} - \delta_0)^{\varrho - 1} \widehat{w}(u)\,du + \int_{1}^{2} (\delta_0 u + \beta_{\pm} - \delta_0)^{\varrho - 1} \widehat{w}(u)\,du \right| \\
    &\leq \frac{\beta_{\pm} - \alpha_{\pm}}{\alpha_{\pm}^{1/2}} + \frac{\delta_0}{(\alpha_{\pm}-\delta_0)^{1/2}} \int_{0}^{2} \widehat{w}(u)\,du \\
    &\leq \frac{h_0}{\alpha_{\pm}^{1/2}} + \frac{\delta_0}{(\alpha_{\pm}-\delta_0)^{1/2}} \\
    &\leq h_0 + \frac{\delta_0}{(\alpha_{\pm}-\delta_0)^{1/2}} . 
\end{align*} 
For every non-trivial zero $\varrho = 1/2 + i\gamma$, so we have $|\varrho + j| \geq |\gamma|$. It follows from this and the computations in \eqref{eqn:integral_comps} that if the RH is true, then 
\begin{align*}
    |W_{\pm}(\varrho)|
    &\leq
    \left\{
    \begin{array}{l}
        h_0 + \frac{\delta_0}{(1-\delta_0)^{m + 1/2}} \\
        |\gamma|^{-1} (1 + h_0 + \delta_0)^{\frac{1}{2}} \int_{0}^{2} |\widehat{w}^{(1)}(u)| \,du \\
        |\gamma|^{-2} \delta_0^{-1} (1 + h_0 + \delta_0)^{\frac{3}{2}} \int_{0}^{2} |\widehat{w}^{(2)}(u)| \,du 
    \end{array}
    \right. 
    \leq
    \left\{
    \begin{array}{l}
        h_0 + \frac{\delta_0}{\sqrt{1-\delta_0}} \\
        |\gamma|^{-1} 2 (1 + h_0 + \delta_0)^{\frac{1}{2}} \\
        |\gamma|^{-2} 2\pi \delta_0^{-1} (1 + h_0 + \delta_0)^{\frac{3}{2}} 
    \end{array}
    \right. .
\end{align*}

To prove the final statement, note that for any positive integer $\ell$, it follows from integration by parts and the substitutions $t = \delta_0 u + \alpha_{\pm} - \delta_0$, $t = \delta_0 u + \beta_{\pm} - \delta_0$ that
\begin{align*}
    |W_{\pm}(-2\ell)| 
    &= \frac{1}{2\ell} \left|\int_{\alpha_{\pm}-\delta_0}^{\beta_{\pm}+\delta_0} t^{-2\ell} w_{\pm}^{(1)}(t)\,dt \right| \\
    &\leq \frac{1}{2\ell} \left|\left(\int_{\alpha_{\pm}-\delta_0}^{\alpha_{\pm}} + \int_{\beta_{\pm}}^{\beta_{\pm}+\delta_0}\right) t^{-2\ell} w_{\pm}^{(1)}(t)\,dt \right| \\
    &\leq \frac{(1-\delta_0)^{-2\ell}}{2\ell} \left|\left(\int_{\alpha_{\pm}-\delta_0}^{\alpha_{\pm}} + \int_{\beta_{\pm}}^{\beta_{\pm}+\delta_0}\right) w_{\pm}^{(1)}(t)\,dt \right| \\
    &= \frac{(1-\delta_0)^{-2\ell}}{2\ell} \left|\int_{0}^{2} \widehat{w}^{(1)}(u)\,du \right| 
    = \frac{(1-\delta_0)^{-2\ell}}{\ell} . \qedhere
\end{align*}
\end{proof}

\subsection{Step 2: An explicit formula}\label{ssec:explicit_formula}

Applying the content of Section \ref{ssec:auxiliary_bounds}, we prove the following result, which is an explicit formula for $\Omega_{\pm}(x)$. 

\begin{lemma}\label{lem:step_2}
We have
\begin{equation*}
    \left|\Omega_{\pm}(x)  - h + \sum_{\varrho} x^{\varrho} W_{\pm}(\varrho)\right| 
    \leq \delta + \frac{1}{(x-\delta)^2} .
\end{equation*}
\end{lemma}

\begin{proof}
It follows from \eqref{eqn:InvMellinTransformW} that
\begin{align*}
    \Omega_{\pm}(x)
    = \sum_{n} \frac{1}{2\pi i} \int_{2-i\infty}^{2+i\infty} x^s W_{\pm}(s) \Lambda(n) n^{-s} \,ds 
    &= - \frac{1}{2\pi i} \int_{2-i\infty}^{2+i\infty} x^s W_{\pm}(s) \frac{\zeta'(s)}{\zeta(s)} \,ds ,
\end{align*}
in which $\zeta(s)$ is the Riemann zeta-function. Using well known properties of $\zeta(s)$, moving the line of integration, and applying the residue theorem, we have
\begin{align}
    \Omega_{\pm}(x)
    &= x W_{\pm}(1) - \sum_{\varrho} x^{\varrho} W_{\pm}(\varrho) - \sum_{k=1}^\infty x^{-2k} W_{\pm}(-2k) , \label{eqn:step_1}
\end{align}
because $W_{\pm}(s)$ is meromorphic with one simple pole at $s=0$ whose residue is $0$, making this a removable singularity.\footnote{A proof of \eqref{eqn:step_1} would be analogous to the proof of \cite[Thm.~2.3]{FaberKadiri}, so we exclude it.} 
Further, Lemma \ref{lem:properties} tells us
\begin{align*}
    \sum_{k=1}^\infty x^{-2k} |W_{\pm}(-2k)|
    \leq 2 \sum_{\ell =1}^\infty \frac{(x-\delta)^{-2\ell}}{2\ell}
    = - \logb{1-\frac{1}{(x-\delta)^2}} 
    &\leq \frac{1}{(x-\delta)^2} .
\end{align*}
With this, the result follows from \eqref{eqn:step_1} and Lemma \ref{lem:properties}. 
\end{proof}

\subsection{Step 3: A parametrised approximate formula}\label{ssec:approximate_formula}

To bound the sum over zeros that is present in Lemma \ref{lem:step_2}, we import the following results. 

\begin{lemma}\label{lem:results_collected}
First, if $T\geq 2\pi$, then the number $N(T)$ of zeros $\varrho = \beta + i\gamma$ of $\zeta(s)$ such that $0<\gamma \leq T$ satisfies the relationship
\begin{align}
    |Q(T)| 
    &:= \left|N(T) - \frac{T}{2\pi} \logb{\frac{T}{2\pi e}} + \frac{7}{8}\right| \nonumber\\
    &\leq \min\{0.28\log{T}, 0.1038\log{T} + 0.2573\log\log{T} + 9.3675\} 
    := R(T) . \label{eqn:NT_est}
\end{align}
A straightforward consequence of \eqref{eqn:NT_est} is that if $T\geq 1$, then
\begin{equation}\label{eqn:use_me_instead}
    N(T) \leq \frac{T\log{T}}{2\pi} .
\end{equation}
Finally, we have
\begin{equation}\label{eqn:Skewes}
    \sum_{|\gamma| \geq T} \frac{1}{\gamma^2} < \frac{\log{T}}{\pi T} 
    \quad\text{for all}\quad T\geq 1.
\end{equation}
\end{lemma}

\begin{proof}
The first result \eqref{eqn:NT_est} is a combination of observations from Brent--Platt--Trudgian \cite[Cor.~1]{BPT_mean_square} and Hasanalizade--Shen--Wong \cite[Cor.~1.2]{HasanalizadeShenWong}. 
Finally, Skewes proved \eqref{eqn:Skewes} in \cite[Lem.~1(ii)]{skewes}.
\end{proof}

\begin{lemma}\label{lem:BPT}
Let $A_0 = 2.067$, $A_1 = 0.059$, and $A_2 = 1/150$. If $2\pi \leq U \leq V$, then
\begin{equation*}
    \sum_{\substack{U \leq \gamma \leq V}} \phi(\gamma)
    = \frac{1}{2\pi} \int_{U}^{V} \phi(t)\logb{\frac{t}{2\pi}}\,dt + \phi(V)Q(V) - \phi(U)Q(U) + \mathcal{E}(U),
\end{equation*}
in which $|\mathcal{E}(U)| \leq 2(A_0 + A_1\log{U})|\phi'(U)| + (A_1 + A_2) \frac{\phi(U)}{U}$ and \eqref{eqn:NT_est} tell us $|Q(T)| \leq R(T)$.
\end{lemma}

\begin{proof}
This result was proved by Brent, Platt, and Trudgian in \cite{BrentAccurate}. It refines a well-known result from Lehman \cite[Lem.~1]{Lehman}.
\end{proof}

Using these auxiliary bounds, we convert Lemma \ref{lem:step_2} into a usable, parametrised approximation for $\Omega_{\pm}(x)$ in the following result. 

\begin{lemma}\label{lem:paramers_everywhere}
Let $1 \leq \kappa_1 \leq \kappa_2$ such that $\kappa_1\kappa_2 \leq 4\pi^2$, $\kappa_3\geq 2$ such that $\kappa_1 \leq \kappa_3 \leq 20$, and $0.1 \leq \varepsilon \leq 0.45$ be real parameters. Further, suppose $\sqrt{x}\log{x} \leq h \leq x^{1-\varepsilon}$, $\delta = \sqrt{x}\log{x} / \kappa_3$, and 
\begin{equation*}
    E(T_1,T_2) = 2\left(\frac{R(T_2)}{T_2} + \frac{R(T_1)}{T_1} + \frac{A_1 + A_2 + 2 (A_0+A_1\log{T_1})}{T_1^2} \right) .
\end{equation*} 
If the RH is true and $\log{x} \geq 20$, then
\begin{equation*}
    |\Omega_{\pm}(x) - h |
    < \left(\frac{1}{\pi} \logb{\frac{h}{\sqrt{x}\log{x}}} + \Psi_1(x, \varepsilon)\right) \sqrt{x}\log{x} , 
\end{equation*}
in which
\begin{align*}
    \Psi_1(x,\varepsilon)
    &= \frac{1}{\kappa_3} 
    + \frac{\kappa_1}{2\pi} \left(1 + \frac{\kappa_3^{-1}}{\sqrt{1 - \delta_0}}\right) + \frac{1}{\kappa_2}\left(1 + \frac{2\log{\kappa_2}}{\log{x}}\right) \left(1 + \frac{3x^{-\varepsilon}}{2}\right)^{\frac{3}{2}} \\
    &\qquad + \left( \frac{1}{\pi} \logb{\frac{\kappa_2 \kappa_3}{\kappa_1}} + \frac{4 E(\kappa_1 x^{\varepsilon},\tfrac{\kappa_2\kappa_3 \sqrt{x}}{\log{x}})}{\log{x}} \right) \left(1 + \frac{3x^{-\varepsilon}}{2}\right)^{\frac{1}{2}} \\
    &\qquad + \frac{x^{-\frac{5}{2}}}{\log{x}}\left(1-\frac{\log{x}}{\kappa_3\sqrt{x}}\right)^{-2} + \frac{1}{\pi} \logb{\frac{x^{\frac{1}{2} - \varepsilon}}{\sqrt{x}\log{x}}} \left(\left(1 + \frac{3x^{-\varepsilon}}{2}\right)^{\frac{1}{2}} - 1\right) . 
\end{align*}
\end{lemma}

\begin{proof}
Since $\kappa_3 \geq 2$, we have $2\delta \leq h$. Using Lemmas \ref{lem:properties} and \ref{lem:step_2}, we have
\begin{align*}
    |\Omega_{\pm}(x) - h | 
    &\leq \delta + x^{\frac{1}{2}} \sum_{\varrho} |W_{\pm}(\varrho)| + \frac{1}{(x-\delta)^2} \\
    &\leq \delta + \frac{\delta}{\sqrt{x}} \sum_{|\gamma | \leq \frac{\kappa_1 x}{h}} \Phi_0 
    + \sqrt{x} \sum_{\frac{\kappa_1 x}{h} \leq |\gamma | \leq \frac{\kappa_2 x}{\delta}} \frac{\Phi_1}{|\gamma |} 
    + \frac{x^{\frac{3}{2}}}{\delta} \sum_{|\gamma | > \frac{\kappa_2 x}{\delta}} \frac{\Phi_2}{\gamma^2} 
    + \frac{1}{(x-\delta)^2} . 
\end{align*}
It follows from \eqref{eqn:use_me_instead} and $\kappa_1 \leq \log{x}$ that
\begin{align*}
    \frac{\delta}{\sqrt{x}} \sum_{|\gamma | \leq \frac{\kappa_1 x}{h}} \Phi_0
    = \frac{\delta}{\sqrt{x}} \left(\frac{h}{\delta} + \frac{1}{\sqrt{1 - \delta_0}}\right) \sum_{|\gamma | \leq \frac{\kappa_1 x}{h}} 1
    &\leq \frac{\kappa_1 \sqrt{x} \logb{\frac{\kappa_1 x}{h}}}{\pi} \left(1 + \frac{\delta / h}{\sqrt{1 - \delta_0}}\right) \\ 
    &\leq \frac{\kappa_1 \sqrt{x} \logb{\frac{\kappa_1 \sqrt{x}}{\log{x}}}}{\pi} \left(1 + \frac{\kappa_3^{-1}}{\sqrt{1 - \delta_0}}\right) \\
    &\leq \frac{\kappa_1 \sqrt{x} \log{x}}{2\pi} \left(1 + \frac{\kappa_3^{-1}}{\sqrt{1 - \delta_0}}\right) .
\end{align*}
Next, Lemma \ref{lem:BPT} with $\phi(t) = t^{-1}$, algebraic manipulation, and $\kappa_1 x / h \geq 2\pi$ imply 
\begin{align}
    \sum_{\frac{\kappa_1 x}{h} < |\gamma| \leq \frac{\kappa_2 x}{\delta}} \frac{1}{|\gamma|}
    &\leq \frac{1}{2\pi} \left(\logb{\frac{\kappa_2 x}{2\pi \delta}}\right)^2 - \frac{1}{2\pi} \left(\logb{\frac{\kappa_1 x}{2\pi h}}\right)^2 + 2 E(\tfrac{\kappa_1 x}{h},\tfrac{\kappa_2 x}{\delta}) \nonumber\\
    &= \frac{1}{2\pi} \left(\left(\logb{\frac{\kappa_2}{2\pi \delta}}\right)^2 - \left(\logb{\frac{\kappa_1}{2\pi h}}\right)^2\right) \nonumber\\
    &\qquad+ \frac{\log{x}}{\pi} \left(\logb{\frac{\kappa_2}{2\pi \delta}} - \logb{\frac{\kappa_1}{2\pi h}}\right) + 2 E(\tfrac{\kappa_1 x}{h},\tfrac{\kappa_2 x}{\delta}) \nonumber\\
    &= \frac{1}{2\pi} \logb{\frac{\kappa_1\kappa_2 x^2}{4\pi^2 h\delta}} \logb{\frac{\kappa_2 h}{\kappa_1 \delta}} + 2 E(\tfrac{\kappa_1 x}{h},\tfrac{\kappa_2 x}{\delta}) .
    \label{eqn:BPT_applied} 
\end{align}
Note that $\kappa_1 x / h \geq \kappa_1 x^{\varepsilon} \geq x^{0.1} \geq 2\pi$, because $\log{x} \geq 20$. Therefore, the RH and \eqref{eqn:BPT_applied} imply
\begin{align*}
    \sqrt{x} &\sum_{\frac{\kappa_1 x}{h} \leq |\gamma | \leq \frac{\kappa_2 x}{\delta}} \frac{\Phi_1}{|\gamma |}  
    \leq \Phi_1 \sqrt{x} \left( \frac{1}{2\pi} \logb{\frac{\kappa_1\kappa_2 x^2}{4\pi^2 h\delta}} \logb{\frac{\kappa_2 h}{\kappa_1 \delta}} + 2 E(\tfrac{\kappa_1 x}{h},\tfrac{\kappa_2 x}{\delta}) \right) \\
    &\leq \Phi_1 \sqrt{x} \left( \frac{1}{2\pi} \logb{\frac{\kappa_1\kappa_2\kappa_3 x}{4\pi^2 (\log{x})^2}} \logb{\frac{\kappa_2 \kappa_3 h}{\kappa_1 \sqrt{x}\log{x}}} + 2 E(\tfrac{\kappa_1 x}{h},\tfrac{\kappa_2\kappa_3 \sqrt{x}}{\log{x}}) \right) \\
    &\leq \Phi_1 \sqrt{x} \left( \frac{\log{x}}{2\pi} \left(\logb{\frac{\kappa_2 \kappa_3}{\kappa_1}} + \logb{\frac{h}{\sqrt{x}\log{x}}}\right) + 2 E(\kappa_1 x^{\varepsilon},\tfrac{\kappa_2\kappa_3 \sqrt{x}}{\log{x}}) \right) \\
    &\leq \sqrt{x} \left( \frac{\log{x}}{\pi} \left(\logb{\frac{\kappa_2 \kappa_3}{\kappa_1}} + \logb{\frac{h}{\sqrt{x}\log{x}}}\right) + 4 E(\kappa_1 x^{\varepsilon},\tfrac{\kappa_2\kappa_3 \sqrt{x}}{\log{x}}) \right) \left(1 + \frac{3h}{2x}\right)^{\frac{1}{2}} ,
\end{align*}
since $\kappa_1\kappa_2 \leq 4\pi^2$ and $\log{x} \geq 20 \geq \kappa_3$. Finally, \eqref{eqn:Skewes} implies
\begin{align*}
    \frac{x^{\frac{3}{2}}}{\delta} \sum_{|\gamma | > \frac{\kappa_2 x}{\delta}} \frac{\Phi_2}{\gamma^2} 
    \leq \frac{\Phi_2 \sqrt{x}\logb{\frac{\kappa_2 x}{\delta}}}{\kappa_2 \pi} 
    &\leq \frac{\Phi_2 \sqrt{x}\logb{\frac{\kappa_2\kappa_3 \sqrt{x}}{\log{x}}}}{\kappa_2 \pi} \\
    &\leq \frac{\Phi_2 \sqrt{x} (2\log{\kappa_2} + \log{x})}{2\pi \kappa_2} \\ 
    &\leq \frac{\sqrt{x}\log{x}}{\kappa_2}\left(1 + \frac{2\log{\kappa_2}}{\log{x}}\right) \left(1 + \frac{3h}{2x}\right)^{\frac{3}{2}} , 
\end{align*}
since $\log{x} \geq 20 \geq \kappa_3$. Combining these observations, we have proved the result. 
\end{proof}

\subsection{Step 4: Final deductions}\label{ssec:final_steps}

Next, we make optimised choices for each of the parameters $\kappa_i$ in Lemma \ref{lem:paramers_everywhere} to deduce the following result.

\begin{lemma}\label{lem:Omega_omega_approximation}
Recall the definition of $L$ from \eqref{eqn:L_def}. Let
\begin{equation*}
    \kappa_1 = \frac{2}{1 + \kappa_3^{-1}} , \quad
    \kappa_2 = \pi , \quad
    \text{and}\quad
    \kappa_3 = \frac{\pi + \sqrt{\pi^2 + 4\pi}}{2} .
\end{equation*}
If the RH is true, $\log{x} \geq 20$, and $\sqrt{x}\log{x} \leq h \leq x^{3/4}$, then 
\begin{equation}\label{eqn:recall_me_omega}
    |\Omega_{\pm}(x) - h |
    < \left(\frac{1}{\pi} \logb{\frac{h}{\sqrt{x}\log{x}}} + L + \frac{\Psi(x)}{\log{x}}\right) \sqrt{x}\log{x} , 
\end{equation}
in which 
\begin{align*}
    \Psi(x)
    &= \frac{\kappa_1 (\log{x})^2}{4\pi\kappa_3^2\sqrt{x}} \left(1 - \frac{\log{x}}{\kappa_3 \sqrt{x}}\right)^{-\frac{3}{2}} 
    + x^{-\frac{5}{2}} \left(1 - \frac{\log{x}}{\kappa_3 \sqrt{x}}\right)^{-2} 
    + \frac{3 \log{x}}{4\pi x^{1/4}} \logb{\frac{\kappa_2 \kappa_3 x^{1/4}}{\kappa_1 \log{x}}} \\
    &\qquad + \left(\frac{9 \log{x}}{4\kappa_2 x^{1/4}} + 4 E(\kappa_1 x^{1/4},\tfrac{\kappa_2\kappa_3 \sqrt{x}}{\log{x}})\right) \left(1 + \frac{3x^{-1/4}}{2}\right)^{\frac{1}{2}} 
    + \frac{2\log{\kappa_2}}{\kappa_2} \left(1 + \frac{3x^{-1/4}}{2}\right)^{\frac{3}{2}} . 
\end{align*}
\end{lemma}

\begin{proof}
To begin, note that
\begin{align*}
    \lim_{x\to\infty} \Psi_1(x,\varepsilon) 
    &= \frac{1}{\kappa_3} 
    + \frac{\kappa_1}{2\pi} \left(1 + \frac{1}{\kappa_3}\right) 
    + \frac{\log{\kappa_2} + \log{\kappa_3} - \log{\kappa_1}}{\pi} 
    + \frac{1}{\kappa_2} .
\end{align*}
We choose $\kappa_1$, $\kappa_2$, and $\kappa_3$ such that this expression is minimised. The terms involving $\kappa_1$ are minimised at $\kappa_1 = 2/(1+\kappa_3^{-1})$ and the terms involving $\kappa_2$ are minimised at $\kappa_2 = \pi$. Inserting these choices, we obtain
\begin{align*}
    \lim_{x\to\infty} \Psi_1(x,\varepsilon) 
    &= \frac{1}{\kappa_3} + \frac{2 + \logb{\frac{\pi}{2}(\kappa_3 + 1)}}{\pi} ,
\end{align*}
which is minimised at $\kappa_3 = (\pi+\sqrt{4\pi+\pi^{2}})/2$. With these choices inserted, $$L = \lim_{x\to\infty} \Psi_1(x,\varepsilon) , $$ which satisfies \eqref{eqn:L_def}. If the RH is true, $0.1 \leq \varepsilon \leq 0.25$, and $\log{x} \geq 20$, then Lemma \ref{lem:paramers_everywhere} tells us 
\begin{equation}\label{eqn:recall_me_omega}
    |\Omega_{\pm}(x) - h |
    < \left(\frac{1}{\pi} \logb{\frac{h}{\sqrt{x}\log{x}}} + L + \Psi_2(x, \varepsilon)\right) \sqrt{x}\log{x} , 
\end{equation}
in which
\begin{align*}
    \Psi_2(x,\varepsilon)
    &= \frac{\kappa_1}{2\pi\kappa_3} \left(\left(1 - \frac{\log{x}}{\kappa_3\sqrt{x}}\right)^{-\frac{1}{2}} - 1\right) + \frac{1}{\pi} \logb{\frac{\kappa_2 \kappa_3 x^{\frac{1}{2}-\varepsilon}}{\kappa_1 \log{x}}} \left(\left(1 + \frac{3x^{-\varepsilon}}{2}\right)^{\frac{1}{2}} - 1\right) \\
    &\qquad + \frac{4 E(\kappa_1 x^{\varepsilon},\tfrac{\kappa_2\kappa_3 \sqrt{x}}{\log{x}})}{\log{x}} \left(1 + \frac{3x^{-\varepsilon}}{2}\right)^{\frac{1}{2}} + \frac{1}{\kappa_2} \left(\left(1 + \frac{3x^{-\varepsilon}}{2}\right)^{\frac{3}{2}} - 1\right) \\
    &\qquad + \frac{2\log{\kappa_2}}{\kappa_2 \log{x}} \left(1 + \frac{3x^{-\varepsilon}}{2}\right)^{\frac{3}{2}} + \frac{x^{-\frac{5}{2}}}{\log{x}} \left(1 - \frac{\log{x}}{\kappa_3\sqrt{x}}\right)^{-2} . 
\end{align*}
To bound each of the differences $(1 \pm o(1))^r - 1$, re-write the difference as an integral and note that the difference is majorised by the length of the integral (namely $o(1)$) multiplied by the maximum value obtained by the integrand. It follows that
\begin{align*}
    \Psi_2(x,\varepsilon) &\leq \frac{\kappa_1\log{x}}{4\pi\kappa_3^2\sqrt{x}} \left(1 - \frac{\log{x}}{\kappa_3 \sqrt{x}}\right)^{-\frac{3}{2}} + \frac{x^{-\frac{5}{2}}}{\log{x}} \left(1 - \frac{\log{x}}{\kappa_3 \sqrt{x}}\right)^{-2} + \frac{3 x^{-\varepsilon}}{4\pi} \logb{\frac{\kappa_2 \kappa_3 x^{\frac{1}{2}-\varepsilon}}{\kappa_1 \log{x}}} \\
    &\qquad + \left(\frac{9 x^{-\varepsilon}}{4\kappa_2} + \frac{4 E(\kappa_1 x^{\varepsilon},\tfrac{\kappa_2\kappa_3 \sqrt{x}}{\log{x}})}{\log{x}}\right) \left(1 + \frac{3x^{-\varepsilon}}{2}\right)^{\frac{1}{2}} + \frac{2\log{\kappa_2}}{\kappa_2 \log{x}} \left(1 + \frac{3x^{-\varepsilon}}{2}\right)^{\frac{3}{2}} .
\end{align*}
With this and $\varepsilon = 1/4$, we have proved the result.
\end{proof}

\begin{proof}[Proof of Theorem \ref{thm:approximation}]
The function $\Psi(x)$ decreases as $x$ increases on $\log{x}\geq 20$. Therefore, assuming the same choices as in Lemma \ref{lem:Omega_omega_approximation}, we have proved that if the RH is true, $\log{x}\geq 20$, and $\sqrt{x}\log{x} \leq h \leq x^{3/4}$, then 
\begin{align*}
    |\Omega_{\pm}(x) - h |
    &< \left(\frac{1}{\pi} \logb{\frac{h}{\sqrt{x}\log{x}}} + L + \frac{\Psi(e^{20})}{\log{x}}\right) \sqrt{x}\log{x} \\
    &\leq \left(\frac{1}{\pi} \logb{\frac{h}{\sqrt{x}\log{x}}} + L + \frac{1.02164}{\log{x}}\right) \sqrt{x}\log{x} ,
\end{align*}
which is the result.
\end{proof}


\section{Proof of main results}\label{sec:proof_pnt_si}

Using Theorem \ref{thm:approximation}, we prove Theorem \ref{thm:main_pnt} in Section \ref{sec:proof_pnt_si}. Building upon this, we prove Theorem \ref{thm:prime_sum_bounder} in Section \ref{sec:proof2}. 

\subsection{Proof of Theorem \ref{thm:main_pnt}}\label{sec:proof_pnt_si}

We prove each statement \eqref{eqn:SI_result_psi_explicit} and \eqref{eqn:SI_result_theta_explicit} in Theorem \ref{thm:main_pnt} independently. To this end, we require the following auxiliary bound.

\begin{lemma}\label{lem:diffs}
If $x\geq e^{20}$ and $\sqrt{x}\log{x} \leq h \leq x^{3/4}$, then $$| \psi(x+h) - \psi(x) - \theta(x+h) + \theta(x) | \leq 0.00437 \sqrt{x} + 0.5445 x^{\frac{1}{3}}.$$
\end{lemma}

\begin{proof}
It follows from \cite[Eqns.~(21),~(31)]{Costa} and \cite[Cor.~5.1]{Broadbent} that for all $\log{x}\geq 20$, we have
\begin{equation}\label{eqn:diffs}
    0.999 x^\frac{1}{2} + 0.885 x^\frac{1}{3}
    < \psi(x)-\theta(x)
    < \alpha_1 x^\frac{1}{2} + \alpha_2 x^\frac{1}{3},
\end{equation}
in which $\alpha_1= 1 + 1.93378 \cdot 10^{-8}$ and $\alpha_2 = 1.4263$. It follows that for all $\log{x}\geq 20$, we have
\begin{align*}
    \psi(x+h) - \psi(x) &- \theta(x+h) + \theta(x) \\
    &< \Bigg( \alpha_1 \left(1 + \frac{h}{x}\right)^{\frac{1}{2}} - 0.999\Bigg) \sqrt{x} 
    + \Bigg( \alpha_2 \left(1 + \frac{h}{x}\right)^{\frac{1}{3}} - 0.885\Bigg) x^{\frac{1}{3}} \\
    &\leq \left( \alpha_1 \left(1 + x^{-\frac{1}{4}}\right)^{\frac{1}{2}} - 0.999\right) \sqrt{x} 
    + \Bigg( \alpha_2 \left(1 + x^{-\frac{1}{4}}\right)^{\frac{1}{3}} - 0.885\Bigg) x^{\frac{1}{3}} \\
    &\leq 0.00437 \sqrt{x} + 0.5445 x^{\frac{1}{3}} 
\end{align*}
and
\begin{align*}
    \psi(x+h) - \psi(x) &- \theta(x+h) + \theta(x) \\
    &\geq \Bigg( 0.999 \left(1 + \frac{h}{x}\right)^{\frac{1}{2}} - \alpha_1\Bigg) \sqrt{x} 
    + \Bigg( 0.885 \left(1 + \frac{h}{x}\right)^{\frac{1}{3}} - \alpha_2 \Bigg) x^{\frac{1}{3}} \\
    &> \Bigg( 0.999 \left(1 + \frac{\log{x}}{\sqrt{x}}\right)^{\frac{1}{2}} - \alpha_1\Bigg) \sqrt{x} 
    + \Bigg( 0.885 \left(1 + \frac{\log{x}}{\sqrt{x}}\right)^{\frac{1}{3}} - \alpha_2 \Bigg) x^{\frac{1}{3}} \\
    &\geq - 0.00055 \sqrt{x} - 0.54104 x^{\frac{1}{3}} .
\end{align*}
Combining these observations we have proved the result.
\end{proof}

\begin{proof}[Proof of \eqref{eqn:SI_result_psi_explicit}]
The definition of $w_{\pm}$ tells us that
\begin{equation}\label{eqn:bounds}
    \Omega_{-}(x) < \psi(x+h) - \psi(x) < \Omega_{+}(x) .
\end{equation}
Using Theorem \ref{thm:approximation}, it follows that if the RH is true, $\log{x} \geq 20$, and $\sqrt{x}\log{x} \leq h \leq x^{3/4}$, then
\begin{equation*}
    \left| \psi(x+h) - \psi(x) - h \right| < \left(\frac{1}{\pi} \logb{\frac{h}{\sqrt{x}\log{x}}} + L + \frac{1.02164 }{\log{x}} \right) \sqrt{x}\log{x} ,
\end{equation*}
which is stronger than \eqref{eqn:SI_result_psi_explicit}.
\end{proof}

\begin{proof}[Proof of \eqref{eqn:SI_result_theta_explicit}]
It follows from \eqref{eqn:SI_result_psi_explicit} and Lemma \ref{lem:diffs} that if the RH is true, $\log{x}\geq 20$, and $\sqrt{x}\log{x} \leq h \leq x^{3/4}$, then
\begin{align*}
    |\theta(x+h) &- \theta(x) - h|
    = |\theta(x+h) - \theta(x) - \psi(x+h) + \psi(x) + \psi(x+h) - \psi(x) - h| \\
    &\leq |\psi(x+h) - \psi(x) - h| + 0.00437 \sqrt{x} + 0.5445 x^{\frac{1}{3}} \\
    &\leq \left(\frac{1}{\pi} \logb{\frac{h}{\sqrt{x}\log{x}}} + L + \frac{1.02164 }{\log{x}}\right) \sqrt{x}\log{x} + 0.00437 \sqrt{x} + 0.5445 x^{\frac{1}{3}} \\
    &\leq \left(\frac{1}{\pi} \logb{\frac{h}{\sqrt{x}\log{x}}} + L + \frac{1.02601  + 0.5445 x^{-\frac{1}{6}}}{\log{x}}\right) \sqrt{x}\log{x} \\
    &\leq \left(\frac{1}{\pi} \logb{\frac{h}{\sqrt{x}\log{x}}} + L + \frac{1.04544}{\log{x}}\right) \sqrt{x}\log{x} ,
\end{align*}
which is stronger than \eqref{eqn:SI_result_theta_explicit}. 
\end{proof}

\subsection{Proof of Theorem \ref{thm:prime_sum_bounder}}\label{sec:proof2}

Recall from \cite[(4.15)]{Rosser} that
\begin{equation}\label{eqn:Kerfuffle}
    \sum_{p\leq x} f(p)
    = \int_2^x \frac{f(t)}{\log{t}}\,dt + L_f + \frac{(\theta(x) - x) f(x)}{\log{x}} 
    + \int_{x}^{\infty} (\theta(t) - t) \frac{d}{dt}\left(\frac{f(t)}{\log{t}}\right)\,dt ,
\end{equation}
where $L_f$ is a constant given by
\begin{equation*}
    L_f = \frac{2 f(2)}{\log{2}} - \int_{2}^{\infty} (\theta(t) - t) \frac{d}{dt}\left(\frac{f(t)}{\log{t}}\right)\,dt .
\end{equation*}
Using this relationship, we prove Theorem \ref{thm:prime_sum_bounder} as follows.

\begin{proof}[Proof of Theorem \ref{thm:prime_sum_bounder}]
Since $L_f$ is a constant, \eqref{eqn:Kerfuffle} implies
\begin{align*}
    \sum_{x < p\leq x+h} f(p) 
    &= \int_{x}^{x+h} \frac{f(t)}{\log{t}}\,dt + \frac{(\theta(x+h) - x - h) f(x+h)}{\log(x+h)} \\
    &\hspace{3cm} - \frac{(\theta(x) - x) f(x)}{\log{x}} 
    - \int_{x}^{x+h} (\theta(t) - t) \frac{d}{dt}\left(\frac{f(t)}{\log{t}}\right)\,dt .
\end{align*}
Since $f(x)$ is non-increasing and $-f(x) \leq f(x+h) - f(x) \leq 0$, we have 
\begin{equation*}
    \Bigg| \frac{f(x+h)}{\log(x+h)} - \frac{f(x)}{\log{x}} \Bigg| 
    \leq \frac{| f(x+h) - f(x)|}{\log{x}}
    \leq \frac{f(x)}{\log{x}} .
\end{equation*}
Therefore, we have
\begin{align*}
    \Bigg| \sum_{x < p\leq x+h} f(p) - \int_{x}^{x+h} \frac{f(t)}{\log{t}}\,dt \Bigg| 
    &\leq \frac{|\theta(x+h) - \theta(x) - h| f(x)}{\log{x}} + \frac{|\theta(x+h) - x - h | f(x)}{\log{x}} \\
    &\qquad\qquad - \int_{x}^{x+h} \frac{|\theta(t) - t|}{\log{t}} \left(f'(t) - \frac{f(t)}{t\log{t}}\right)\,dt .
\end{align*}
Note that each of the terms in this bound is positive, since $-f'(t) \geq 0$. If the RH is true and $x\geq 599$, then Schoenfeld \cite{Schoenfeld} also tells us
\begin{equation}\label{eqn:Schoenfeld}
    |\theta(x) - x| < \frac{\sqrt{x}(\log{x})^2}{8\pi} .
\end{equation}
With this, if the RH is true and $\log{x}\geq 20$, then \eqref{eqn:SI_result_theta_explicit} and \eqref{eqn:Schoenfeld} imply
\begin{align*}
    \Bigg| \sum_{x < p\leq x+h} f(p) - \int_{x}^{x+h} \frac{f(t)}{\log{t}}\,dt \Bigg| 
    &\leq \left(\frac{1}{\pi} \logb{\frac{h}{\sqrt{x}\log{x}}} + L + \frac{1.046}{\log{x}}\right) f(x) \sqrt{x} \\
    &\qquad\qquad + \frac{\sqrt{x+h}(\log(x+h))^2}{\sqrt{x} (\log{x})^2} \frac{f(x) \sqrt{x}\log{x}}{8\pi} \\
    &\qquad\qquad - \frac{1}{8\pi} \int_{x}^{x+h} \sqrt{t}\log{t} \left(f'(t) - \frac{f(t)}{t\log{t}}\right)\,dt .
\end{align*}
Finally,
\begin{align*}
    \frac{\sqrt{x+h}(\log(x+h))^2}{\sqrt{x} (\log{x})^2}
    = \left(1 + \frac{h}{x}\right)^{\frac{1}{2}} \left(\frac{\log(x+h)}{\log{x}}\right)^2
    &= \left(1 + \frac{h}{x}\right)^{\frac{1}{2}} \left(1 + \frac{\log(1+h/x)}{\log{x}}\right)^2 \\
    &\leq \left(1 + \frac{h}{x}\right)^{\frac{1}{2}} \left(1 + \frac{h}{x\log{x}}\right)^2 \\
    &\leq \left(1 + \frac{h}{x}\right)^{\frac{5}{2}}.
\end{align*}
The result follows naturally.
\end{proof}

Finally, we demonstrate how Theorem \ref{thm:prime_sum_bounder} can be used to prove \eqref{eqn:MT2_SI_explicit}. 

\begin{proof}[Proof of \eqref{eqn:MT2_SI_explicit}]
Suppose that $f(p) = 1/p$, the RH is true, $\sqrt{x}\log{x} \leq h \leq x^{3/4}$, and $\log{x}\geq 20$. Theorem \ref{thm:prime_sum_bounder} implies
\begin{align*}
    \Bigg| \sum_{x < p\leq x+h} \frac{1}{p} - \int_{x}^{x+h} \frac{1}{t\log{t}}\,dt \Bigg| 
    &\leq \left(\frac{1}{\pi} \logb{\frac{h}{\sqrt{x}\log{x}}} + L + \frac{1.046}{\log{x}}\right) x^{-\frac{1}{2}} \\
    &\qquad\qquad + \left(1 + \frac{h}{x}\right)^{\frac{5}{2}} \frac{\log{x}}{8\pi\sqrt{x}} 
    + \frac{1}{8\pi} \int_{x}^{x+h} \frac{\log{t} + 1}{t^{3/2}} \,dt .
\end{align*}
Further,
\begin{align*}
    \frac{1}{8\pi} \int_{x}^{x+h} \frac{\log{t} + 1}{t^{3/2}} \,dt
    &= \frac{1}{8\pi} \left( \frac{2(\log(x+h) + 3)}{\sqrt{x+h}} - \frac{2(\log{x} + 3)}{\sqrt{x}} \right) \\
    &\leq \frac{2(\log(x+h) - \log{x})}{8\pi\sqrt{x}}
    = \frac{2 \logb{1+\frac{h}{x}}}{8\pi\sqrt{x}}
    \leq \frac{2h}{8\pi x^{3/2}} .
\end{align*}
Therefore, we have proved
\begin{align*}
    &\Bigg| \sum_{x < p\leq x+h} \frac{1}{p} - \int_{x}^{x+h} \frac{1}{t\log{t}}\,dt \Bigg| \\
    &\qquad\qquad\leq \left(\frac{1}{\pi} \logb{\frac{x^{1/4}}{\log{x}}} + L + \frac{1.046}{\log{x}}\right) x^{-\frac{1}{2}} + \left(1 + \frac{h}{x}\right)^{\frac{5}{2}} \frac{\log{x}}{8\pi\sqrt{x}} + \frac{2h}{8\pi x^{3/2}} \\
    &\qquad\qquad\leq \left(\frac{1}{4\pi} + \frac{1}{8\pi} \left(1 + \frac{h}{x}\right)^{\frac{5}{2}} + \frac{1}{\log{x}} \left(L + \frac{1.046}{\log{x}}\right)\right) \frac{\log{x}}{\sqrt{x}} + \frac{2h}{8\pi x^{3/2}} \\ 
    &\qquad\qquad= \left(2 + \left(1 + \frac{h}{x}\right)^{\frac{5}{2}} + \frac{8\pi}{\log{x}} \left(L + \frac{1.046}{\log{x}}\right) + \frac{2h}{8\pi x\log{x}}\right) \frac{\log{x}}{8\pi\sqrt{x}}
    \ll \frac{3\log{x}}{8\pi\sqrt{x}} . \qedhere
\end{align*}
\end{proof}

\section*{Acknowledgements}

The author thanks the Heilbronn Institute for Mathematical Research for their support. Further, I am fortunate to be surrounded by excellent colleagues who have provided valuable feedback and support throughout the development of this paper.

\bibliographystyle{amsplain}
\bibliography{references}

\end{document}